\documentclass[12pt]{amsart}
\usepackage{amsmath,amssymb,amsthm,amscd}
\usepackage{graphicx}
\usepackage{graphics}
\usepackage{color}
\usepackage{enumitem}
\usepackage[plainpages=false]{hyperref}

\newcommand{\beq}{\begin{equation}}
\newcommand{\eeq}{\end{equation}}
\newcommand{\beqs}{\begin{eqnarray*}}
	\newcommand{\eeqs}{\end{eqnarray*}}
\newcommand{\beqn}{\begin{eqnarray}}
\newcommand{\eeqn}{\end{eqnarray}}
\newcommand{\beqa}{\begin{array}}
	\newcommand{\eeqa}{\end{array}}

\newcommand{\sq}{{\sqrt{-1}}}

\newcommand{\bC}{{\mathbb C}}
\newcommand{\bR}{{\mathbb R}}

\newcommand{\lp}{\left(}
\newcommand{\rp}{\right)}
\newcommand{\lf}{\left[}
\newcommand{\rf}{\right]}
\newcommand{\ls}{\left\{}
\newcommand{\rs}{\right\}}

\newcommand{\p}{\partial}
\newcommand{\eps}{\varepsilon}

\newcommand{\pbp}{\p\bar \p}

\newcommand\C{{\mathbb C}}

\DeclareMathOperator{\real}{Re}

\newtheorem{prop}{Proposition}[section]
\newtheorem{theo}[prop]{Theorem}
\newtheorem{lem}[prop]{Lemma}

\newtheorem{rem}[prop]{Remark}
\newtheorem{ex}[prop]{Example}
\newtheorem{defi}[prop]{Definition}

\begin{document}
 
\title[Singular Solutions to the Complex Monge-Amp\`ere Equation]
{Singular Solutions to the Complex Monge-Amp\`ere equation}

\author{Jiaxiang Wang}
\address [Jiaxiang Wang] {School of Mathematical Sciences and LPMC, Nankai University, Tianjin, 300071, P. R. China}
\email{wangjx\underline{ }manifold@126.com}

\author{Wenlong Wang}
\address [Wenlong Wang] {School of Mathematical Sciences and LPMC, Nankai University, Tianjin, 300071, P. R. China}
\email{wangwl@nankai.edu.cn}

\thanks{Jiaxiang Wang is partially supported by the Fundamental Research Funds for the Central Universities Nankai University (050-63241427). Wenlong Wang is partially supported by NSFC 12071338, 12001292, and the Fundamental Research Funds for the Central Universities Nankai University (050-63243065).}

\begin{abstract}
We present an explicit pluripotential and viscosity solution to the complex Monge-Amp\`ere equation with constant right-hand side on $\mathbb D\times\mathbb C^{n-1}\,(n\geq 2)$, which lies merely in $W^{1,2}_{loc}\cap W^{2,1}_{loc}$ and is not even Dini continuous. Additionally, we exhibit two families of explicit entire toric solutions on $\mathbb C^n$ with continuous H\"older exponent $\alpha\in(0,1)$ and $W^{2,p}_{loc}$ exponent $p\in (1,2)$.

\vspace{0.15cm}

\vspace{0.15cm}
\end{abstract}
\maketitle

\section{Introduction}

Complex Monge-Amp\`ere equations play an important role in complex analysis and geometry. There are two fundamental methods for generalizing the notion of solution to functions with low regularity. 
One is the pluripotential approach founded by Bedford and Taylor \cite{BT}, while the other is the viscosity approach introduced by Crandall and Lions for nonlinear PDEs \cite{CL} (see the well-known survey \cite{CIL}) and developed for the complex Monge-Amp\`ere equation in \cite{EGZ,GLZ,HL09,HL13,WangY}. 

The regularity of such generalized solutions is a central topic in the study of the complex Monge-Amp\`ere equation and has been extensively researched. In this note, we investigate the counterpart of this subject: how singular a pluripotential and viscosity solution can be. Our focus lies on examining the complex Monge-Amp\`ere equation defined on a domain of $\mathbb C^n$ in its most elementary form:
\begin{equation}\label{maineq}
\det\left(\partial\bar\partial u\right)=1,
\end{equation}
where $\partial\bar\partial u$ denotes the complex Hessian of $u$. It remains an open question whether a generalized solution possesses interior H\"older regularity. For two-dimensional solutions, even the validity of the interior gradient estimate is unknown. 

We first show that a pluripotential and viscosity solution to \eqref{maineq} does not need to be Dini continuous. 
\begin{theo}\label{main1}
There exist pluripotential and viscosity solutions to \eqref{maineq} on $\mathbb D\times\mathbb C^{n-1}$ that lie merely in $W^{1,2}_{loc}\cap W^{2,1}_{loc}$, but not in $W^{1,p}_{loc}$ for any $p>2$ or $W^{2,q}_{loc}$ for any $q>1$. Moreover, these solutions are not even Dini continuous. 
\end{theo}
Here, $\mathbb D$ denotes the unit disc centered at the origin in $\mathbb C$. We say a function is Dini continuous if the modulus of continuity $\omega(r)$ satisfies $$\int^1_0\frac{\omega(r)}{r}\,dr<\infty.$$
Clearly, H\"older continuity implies Dini continuity. Therefore, Theorem \ref{main1} indicates that pure interior H\"older estimate fails for the complex Monge-Amp\`ere equation. This theorem also shows that the equation provides almost no regularity improvement for the gradient of the solution without additional regularity assumptions, as a bounded plurisubharmonic function naturally lies in $W^{1,2}_{loc}$. For the Hessian, it remains unclear whether $L^1_{loc}$ is the minimal regularity, as is the case for the real Monge-Amp\`ere equation (cf. \cite{DF,Moo1,Moo2}). 

In complex dimension $2$, the solution can be taken as the following explicit function:  
\begin{ex}\label{exa1}
\begin{equation}\label{count-1}
u(z_1,z_2)=-\frac{\left|\real z_2\right|^2}{\log |z_1|}-2|z_1|^2\left(\log |z_1|-1\right).
\end{equation}
\end{ex}
In Section \ref{sec3}, we will verify that Example \ref{exa1} is a solution to \eqref{maineq} in the pluripotential as well as the viscosity sense, and provide its higher-dimensional generalization.  

Besides Example \ref{exa1}, we also construct examples that are toric and global on $\bC^n$, i.e., functions satisfying
$u(z_1,\cdots,z_n)=u(e^{\sq t_1}z_1,\cdots,e^{\sq t_n}z_n)$ for any $(z_1,\cdots,z_n)\in\mathbb C^n$ and $(t_1,\cdots,t_n)\in\mathbb R^n$.

\begin{theo}\label{main2}
There are families of entire toric solutions on $\mathbb C^n$ with continuous H\"older exponent $\alpha\in(0,1)$ and $W^{2,p}$ exponent $p\in (1,2)$.  
\end{theo}
In complex dimension $2$, two families of such solutions can be given as follows.
\begin{ex}\label{exa2}
\begin{equation*}
v(z_1,z_2)=\frac{2}{1-\beta}|z_2|\left(|z_1|^{\beta}+|z_1|^{2-\beta}\right),\quad \beta\in [0,1).
\end{equation*}
\end{ex}
\begin{ex}\label{exa3}
\begin{equation*}
w(z_1,z_2)=|z_1|^{\gamma}|z_2|^2+\frac{4}{(2-\gamma)^2}|z_1|^{2-\gamma},\quad\gamma\in [0,2).
\end{equation*}	
\end{ex}
In Section \ref{sec4}, we will verify that these are solutions to \eqref{maineq} in both the pluripotential and the viscosity sense, and provide higher-dimensional generalizations of Examples \ref{exa2} and \ref{exa3}.  

It is straightforward to see that $v\in C^{\beta}_{loc}\cap W^{1,q}_{loc}\cap W^{2,\,p}_{loc}\,(q<\frac{2}{1-\beta},\,p<\frac{2}{2-\beta})$. Similarly, when $0\leq\gamma\leq 1$, $w\in C^{\gamma}_{loc}\cap W^{1,q}_{loc}\cap W^{2,p}_{loc}\,(q<\frac{2}{1-\gamma},\,p<\frac{2}{2-\gamma})$, and when $1<\gamma<2$, $w\in C^{2-\gamma}_{loc}\cap W^{1,q}_{loc}\cap W^{2,p}_{loc}\,(q<\frac{2}{\gamma-1},\,p<\frac{2}{\gamma})$. Since these two families of solutions include ones with arbitrarily small H\"older exponents, Examples \ref{exa2} and \ref{exa3} show that there is no interior H\"older estimate for equation \eqref{maineq}, even under the toric symmetry and entirety conditions. 

Before our examples, entire solutions on $\mathbb C^n$ with different singularity profiles were constructed by B\l{}ocki \cite{Blo1} and He \cite{He}, respectively. B\l{}ocki's solution is $v(z)=\frac{n}{n-1}\left(1+|z_1|^2\right)|z'|^{2-\frac{2}{n}}$, where $z'=(z_2,\cdots, z_n)$. It is both a pluripotential and viscosity solution to $\det(\partial\bar\partial u)=\left(1+|z_1|^2\right)^{n-2}$ (strong solution for $n\geq 3$). This example, together with B\l{}ocki-Dinew's work \cite{BD}, jointly show that $W_{loc}^{2,n(n-1)}$ is the critical Sobolev regularity needed to guarantee the higher regularity of the solutions. He \cite{He} provided the example $w(z)=n^\frac{2}{n}|z_1|^{\frac{2}{n}}\lp 1+|z'|^2\rp$, which also solves equation \eqref{maineq} both in the viscosity and the pluripotential sense. For dimensions $n\geq 3$, He's solution is non-Lipschitz, indicating that the interior gradient estimate for the complex Monge-Amp\`ere equation fails. Our examples complement this fact for dimension $2$. Note that B\l{}ocki's solution, He's solution and our solutions for $\beta=0$ and $\gamma=1$ essentially coincide in complex dimension $2$. 

These singular solutions, with singularities that extend to the boundary or even to infinity, can potentially be excluded by imposing suitable boundary or global conditions. We review some pertinent results in this context, along with some classical findings.

For a bounded strongly pseudoconvex domain $\Omega\subset\mathbb C^n$, Bedford and Taylor \cite{BT} established the existence and uniqueness of pluripotential solutions for the Dirichlet problem with continuous density and boundary data. For $C^2$ boundary data and density $f$ satisfying $f^\frac{1}{n}\in C^{0,1}(\overline\Omega)$, they further obtained global Lipschitz regularity. Caffarelli, Kohn, Nirenberg, and Spruck \cite{CKNS} proved the existence of classical solutions for the non-degenerate case under suitable conditions on the density and boundary data. This result was later generalized by Guan \cite{Guan}, who extended it to domains admitting a $C^2$ subsolution by refining the barrier arguments. In the degenerate case, Krylov obtained the global $C^{1,1}$ regularity of the solutions, relying on the $C^{3,1}$ boundary data through probabilistic methods in a series of articles (see notably \cite{Krylov}). 

In the seminal work \cite{K98}, Ko\l{}odziej established the $L^{\infty}$ estimate and the existence of a continuous solution under the much milder assumption that the density $f$ belongs to $L^p(\Omega)$ for some $p>1$. For $p=2$, the $L^{\infty}$ estimate and existence theorem are due to Cheng-Yau and Cegrell-Persson \cite{CP}, respectively. Under a mild technical assumption, Guedj, Ko\l odziej, and Zeriahi \cite{GKZ} proved that the solution is H\"older continuous if the boundary data is H\"older continuous and $f\in L^p(\Omega)$ ($p>1$). This result was later generalized and improved in \cite{Cha1,BKPZ,Cha2}.

Li, Li, and Zhang \cite{LLZ} proved interior $C^{\alpha}$ regularity for a solution in a ball with $C^{\alpha}$ density and boundary value. Our Examples \ref{exa2} and \ref{exa3} show that the boundary H\"older regularity condition is necessary, and their result is sharp for the exponent. Cheng and Xu \cite{CX} proved an interior $W^{2,p}$ estimate for the solutions with vanishing boundary values, provided $\Omega$ is close to a ball and $f$ is a small perturbation of a constant.

Examples \ref{exa1}, \ref{exa2} and \ref{exa3} show that there are ``continuous"  singularity profiles lying between the critical regularity needed for self-improvement to smoothness and the near-minimal regularity required for defining the complex Monge-Amp\`ere operator.  Similar “dense” singularity profiles exist for the real Monge-Amp\`ere operator (cf. \cite{CY}).

Compared to the singular solutions of the real Monge-Amp\`ere equation, the much more singular profiles of our solutions underscore the elusiveness of the complex Monge-Amp\`ere operator. A notable characteristic of the singular solutions by B\l{}ocki, He, and ourselves is their explicit form, which may be attributed to the ``double divergence structure" of the complex Monge-Amp\`ere operator in complex dimension two.

Our examples are inspired by Pogorelov's counterexample \cite{Pog} and its various generalizations \cite{Blo1,He,War1,War2}, particularly \cite{War2}, where Warren considered solutions independent of the imaginary part of a complex variable. 

{\it Organization.} In Section \ref{sec2}, we introduce the concepts of pluripotential solution and viscosity solution to the complex Monge-Amp\`ere equation. In Section \ref{sec3}, we prove Theorem \ref{main1} via Example \ref{exa1}, followed by its higher-dimensional generalization. In Section \ref{sec4}, we provide higher-dimensional generalizations of Examples \ref{exa2} and \ref{exa3} and prove Theorem \ref{main2} through them.

\section{Preliminary}\label{sec2}

Let us first briefly recall the construction of Bedford and Taylor \cite{BT}. On $\mathbb C^n$, write $d=\p+\bar \p$ and $d^c=\sqrt{-1}(\bar \p-\p)$ so that $dd^c=2\sqrt{-1}\partial\bar\partial$. For a domain $\Omega\subset \bC^n$, denote by $PSH(\Omega)$ the plurisubharmonic functions on $\Omega$. For $u\in PSH(\Omega)$, $dd^c u$ is a closed positive $(1,1)$ current, i.e., a closed positive $(1,1)$ form with distribution coefficients. If $u$ also lies in $L^\infty_{loc}(\Omega)$, then the complex Monge-Amp\`ere operator can be defined on $u$ inductively as
$$\lp dd^c u\rp^n:=dd^c\left(u\,dd^c\left(u\cdots dd^c  u\right)\right).$$
In particular, if $u\in C^{2}(\Omega)$, then $dd^c u=2\sqrt{-1}\frac{\p^2 u}{\p z^i\p \bar z^j}\,dz^i\wedge d\bar z^j$, and hence 
$$\lp dd^c u\rp^n=4^nn!\det\left(\partial\bar\partial u\right)dV_{\mathbb C^n},$$
where $dV_{\mathbb C^n}$ denotes the standard volume form of $\mathbb C^n$.

In verifying the pluripotential solution in the next two sections, we will use the following weak convergence theorem for the Monge-Amp\`ere measure. 

\begin{lem}[convergence theorem \cite{BT2}]\label{convergence theorem}
Suppose that $\{u_j\}$ is a sequence of functions in $PSH(\Omega)\cap L^{\infty}_{loc}(\Omega)$ that converges decreasingly to $u\in PSH(\Omega)\cap L^{\infty}_{loc}(\Omega)$. Then $\lp dd^c u_j\rp^n$ converges to $\lp dd^c u\rp^n$ in the weak sense of currents on $\Omega$. 
\end{lem}

Another approach to defining the complex Monge-Amp\`ere operator for non-smooth functions is through the viscosity method. 
\begin{defi}[\cite{CIL,EGZ}]\label{vissol}
Let $\Omega$ be a domain of $\C^n$ and $f$ a nonnegative continuous function on $\Omega$. Denote the space of $n$-th order Hermitian matrices by $\mathcal H_n$ and define $F:\Omega\times\mathcal H_n\rightarrow\mathbb R$ by
\begin{equation*}
F(x,A)=\begin{cases}
f(x)-\det(A),\ \ \ &\text{if }A\ge 0;\\
+\infty,   &\text{otherwise.}
\end{cases}
\end{equation*}
\begin{enumerate}[leftmargin=*]
\item [1.] An upper semi-continuous function $u: \Omega\to \bR$ is said to be a viscosity subsolution to the equation $\det(\pbp u)=f$ if for any $p\in \Omega$ and any $C^2$ function $\varphi$ defined in a neighborhood of $p$ such that $u-\varphi$ attains its local maximum at $p$, there holds 
 $$F\left(\pbp \varphi(p)\right)\le 0.$$
\item [2.] A lower semi-continuous function $u: \Omega\to \bR$ is said to be a viscosity supersolution to the equation $\det(\pbp u)=f$ if for any $p\in \Omega$ and any $C^2$ function $\varphi$ defined in a neighborhood of $p$ such that $u-\varphi$ attains its local minimum at $p$, there holds 
 $$F\left(\pbp \varphi(p)\right)\ge 0.$$
\end{enumerate}

We say $u$ is a viscosity solution to $\det(\pbp u)=f$ if it is both a viscosity subsolution and supersolution to this equation.  
\end{defi}
\begin{rem}
Eyssidieux, Guedj and Zeriahi \cite{EGZ} proved that the subsolutions to $\det(\pbp u)=f$ in the viscosity sense coincide with those in the pluripotential sense; a function $u$ on $\Omega$ is plurisubharmonic if and only if it is a viscosity subsolution to the equation $\det(\pbp u)=0$.
\end{rem}

\section{Counterexamples to the interior Dini continuity}\label{sec3}

In this section, we first prove that Example \ref{exa1} is both a pluripotential and viscosity solution to $\det(\partial\bar\partial u)=1$ on $\mathbb D\times\mathbb C$ and verify its singularity profile stated in Theorem \ref{main1}. Then, we provide its higher-dimensional generalization. 

\begin{proof}[Proof of Theorem \ref{main1}] Recall the expression of Example \ref{exa1}: 
\begin{equation*}
u(z_1,z_2)=-\frac{\left|\real z_2\right|^2}{\log |z_1|}-2|z_1|^2\left(\log |z_1|-1\right).
\end{equation*}
Clearly, $u$ is smooth on $\left(\mathbb D\setminus\{0\}\right)\times\mathbb C$. At $\{0\}\times\mathbb C$, define $u=0$ to ensure continuity. The proof is then divided into three parts. 

{\it 1.\,Verification of pluripotential solution.}  

By Lemma \ref{convergence theorem}, we prove this through constructing smooth approximation. For $\varepsilon\in (0,1)$, define a family of smooth functions
\begin{equation*}
u^{\varepsilon}(z_1,z_2)=-\frac{2\left|\real z_2\right|^2}{\log \lp |z_1|^2+\varepsilon\rp}-(|z_1|^2+\varepsilon)\lf \log \lp |z_1|^2+\varepsilon \rp-2\rf
\end{equation*}
on $\mathbb D_{1-\varepsilon}\times\mathbb C$, where $\mathbb D_{1-\varepsilon}$ denotes the disc of radius $1-\varepsilon$ centered at the origin in $\mathbb C$. For any compact subset $\Omega$ of $\mathbb D\times\mathbb C$, $u^{\varepsilon}$ converges decreasingly and uniformly to $u$ on $\Omega$ as $\varepsilon\to 0$. Next we show that $u^{\varepsilon}$ is plurisubharmonic and $\det\lp \pbp u^{\varepsilon}\rp$ converges to $1$ in $L^1(\Omega)$ as $\varepsilon\rightarrow 0$. For simplicity, we use the notation
$$u^{\varepsilon}_k=\frac{\p u^{\varepsilon}}{\p z_k},\quad u^{\varepsilon}_{\bar k}=\frac{\p u^{\varepsilon}}{\p\bar z_k},\quad u^{\varepsilon}_{i\bar j}=\frac{\p^2 u^{\varepsilon}}{\p z_i\p \bar z_j},\quad \text{etc}.$$
 A straightforward computation yields
 \begin{equation*}
u^{\varepsilon}_{1}=\frac{2\left| \real z_2\right|^2 \bar z_1}{\left(|z_1|^2+\varepsilon\right)\log^2 \lp |z_1|^2+\varepsilon\rp}-\lf \log\lp |z_1|^2+\varepsilon\rp-1\rf \bar z_1,\quad\  u^{\varepsilon}_{2}=-\frac{2\real z_2}{\log \lp |z_1|^2+\varepsilon\rp};
\end{equation*}
\begin{equation*}
\begin{split}
u^{\varepsilon}_{1\bar 1}=&-\frac{4|\real z_2|^2|z_1|^2}{\lp |z_1|^2+\varepsilon\rp^2\log^3\lp |z_1|^2+\varepsilon\rp}-\log\lp |z_1|^2+\varepsilon\rp\\
&+\varepsilon\left[\frac{2|\real z_2|^2}{\lp |z_1|^2+\varepsilon\rp^2\log^2\lp |z_1|^2+\varepsilon\rp}+\frac{1}{|z_1|^2+\varepsilon}\right],
\end{split}
\end{equation*}
\begin{equation*}
u^{\varepsilon}_{1\bar 2}=\frac{2(\real z_2)\cdot \bar z_1}{\lp |z_1|^2+\varepsilon\rp\log^2 \lp |z_1|^2+\varepsilon\rp},\qquad 
u^{\varepsilon}_{2\bar 2}=-\frac{1}{\log\lp |z_1|^2+\varepsilon\rp}.
\end{equation*}
We arrive at 
\begin{equation*}
\det(\partial\bar\partial u^{\varepsilon})= 1-\varepsilon\left[\frac{2|\real z_2|^2}{\lp |z_1|^2+\varepsilon\rp^2\log^3\lp |z_1|^2+\varepsilon\rp}+\frac{1}{\lp |z_1|^2+\varepsilon\rp\log\lp |z_1|^2+\varepsilon\rp}\right].
\end{equation*}
From the expression of $\partial\bar\partial u^{\varepsilon}$, we see that $u^\varepsilon$ is plurisubharmonic. Since $u^{\varepsilon}$ converges decreasingly and locally uniformly to $u$, $u$ inherits this property. 

There exist $a\in (0,1)$ and $b>0$ such that $\Omega\subset\mathbb D_{a}\times\mathbb D_{b}$. Let $F_{\varepsilon}=\det(\partial\bar\partial u^{\varepsilon})-1$. Then we have
\begin{equation*}
\begin{split}
\int_{\Omega}\left|F_\varepsilon\right|\,dV_{\mathbb C^2}&\leq C\varepsilon\int_{\mathbb D_a}\frac{1}{\lp |z_1|^2+\varepsilon\rp^2\left[-\log\lp |z_1|^2+\varepsilon\rp\right]^3}\,dV_{\mathbb C}\\
&=C\varepsilon\int^{a^2+\varepsilon}_{\varepsilon}\frac{ds}{s^2(-\log s)^3},
\end{split}
\end{equation*}
where $C$ is a constant independent of $\varepsilon$. Applying L'H\^ospital's rule, we see that the above integral has order $\log^{-3}\varepsilon^{-1}$ as $\varepsilon\rightarrow 0$. Therefore, $F_{\varepsilon}\to 0$ in $L^1(\Omega)$ (in fact in $L^{1}(\log L^1)^p$ for any $p<3$). Hence, $u$ is a pluripotential solution to $\det(\partial\bar\partial u)=1$ on $\mathbb D\times\mathbb C$.

{\it 2.\,Verification of viscosity solution}

We directly prove that $u$ is a viscosity solution according to Definition \ref{vissol}. For $n$-th order Hermitian matrices, define
\begin{equation}\label{F}
F(A):=\begin{cases}
1-\det(A),\ \ &\text{if }A\ge 0;\\
+\infty,   &\text{otherwise.}
\end{cases}
\end{equation}
Based on the calculations for $u^\varepsilon$, we see that $u$ classically solves $\det(\partial\bar\partial u)=1$ on $\left(\mathbb D\setminus\{0\}\right)\times\mathbb C$. Therefore, we only need to check points in $\{0\} \times\mathbb C$, where $u$ takes value $0$. 

Suppose $p_0=(0,z^0_2)$ and $\varphi$ is a $C^2$ function defined in a neighborhood of $p_0$ such that $u-\varphi$ attains its local maximum. Then, for any $p=(z_1,z_2)$ in this neighborhood of $p_0$, we have
$$\varphi(z_1,z_2)\geq \varphi(0,z^0_2)-\frac{|\real z_2|^2}{\log |z_1|}+2|z_1|^2\left(1-\log |z_1|\right).$$
Taking $z_2=z^0_2$, we get
\begin{equation*}
\begin{split}
\varphi(z_1,z^0_2)&\geq \varphi(0,z^0_2)-\frac{|\real z^0_2|^2}{\log |z_1|}+2|z_1|^2\left(1-\log |z_1|\right)\\
&\geq\varphi(0,z^0_2)+ 2|z_1|^2\left(1-\log |z_1|\right).
\end{split}
\end{equation*}
However, this inequality contradicts the assumption that $\varphi\in C^2$. Therefore, there is no upper test function, which proves that $u$ is a viscosity subsolution.  

Suppose $p_0=(0,z^0_2)$ and $\varphi$ is a $C^2$ function defined in a neighborhood of $p_0$ such that $u-\varphi$ attains its local minimum. Then, for any $p=(z_1,z_2)$ in this neighborhood of $p_0$, we have
$$\varphi(z_1,z_2)\leq \varphi(0,z^0_2)-\frac{|\real z_2|^2}{\log |z_1|}+2|z_1|^2\left(1-\log |z_1|\right).$$
Letting $z_1\rightarrow 0$, we obtain
\begin{equation*}
\varphi(0,z_2)\leq \varphi(0,z^0_2).
\end{equation*}
This implies $\varphi_{2\bar 2}(p_0)\leq 0$. If $\pbp\varphi(p_0)\geq 0$, then $\pbp\varphi(p_0)$ has a zero eigenvalue. Thus $\det (\pbp\varphi(p_0))=0$ and $F(\pbp\varphi(p_0))=1$. Otherwise, $F(\pbp\varphi(p_0))=+\infty$. In either case, $F(\pbp\varphi(p_0))\geq 0$. This proves that $u$ is a viscosity supersolution.  

{\it 3.\,Regularity of $u$.} 

Clearly, the modulus of continuity $\omega(r)$ has the order $\log^{-1} r$ as $r\rightarrow 0$. Since $r^{-1}\log^{-1} r$ is not integrable near $0$, $u$ is not Dini continuous.  

A simple calculation shows that $|\nabla u |$ has the order $|z_1|^{-1}\log^{-2}|z_1|$ and  $|\nabla^2 u |$ has the order $|z_1|^{-2}\log^{-2}|z_1|$ as $z_1\rightarrow 0$. This implies that $u$ lies in $W^{1,2}_{loc}\cap W^{2,1}_{loc}$ but not in $W^{1,p}_{loc}$ for any $p>2$ or $W^{2,q}_{loc}$ for any $q>1$. 

\end{proof}

A higher-dimensional generalization of Example \ref{exa1} can be constructed as follows. 
Firstly, consider the following singular ODE:
\begin{align*}
\begin{cases}
rh_n''+h_n'=(-2)^{n+1}r\log^{n-1}r,\ \   r>0.  \\
h_n(0)=0.
\end{cases}
\end{align*}
The solution exists in $[0,1)$ and has the explicit form 
$$h_n(r)=r^2\left(a_{n,n-1}\log^{n-1}r+a_{n,n-2}\log^{n-2}r+\cdots+a_{n,1}\log r+a_{n,0}\right),$$
where $$a_{n,k}=(-2)^{k}\left(\frac{n!}{k!}-\frac{(n-1)!}{(k-1)!}\right)\,\left(k\geq 1\right),\quad a_{n,0}=n!.$$
In particular, $h_2(r)=-2r^2\lp \log r-1\rp$, $h_3(r)=2r^2(2\log^2r-4\log r+3)$. 
As
\begin{equation*}
h_n'(r)=\frac{2^{n+1}}{r}\int^r_{0}s\left(-\log s\right)^{n-1}\,ds>0,
\end{equation*}
$h_n$ monotonically increases in $[0,1)$. For $z=(z_1,z_2,\cdots,z_n)\in \bC^{n}$, write $z'=(z_2,\cdots,z_n)$ and $\real z'=\lp \real z_2,\cdots,\real z_n\rp$. The higher-dimensional version of Example \ref{exa1} defined on $\mathbb D\times\mathbb C^{n-1}$ is the following: 
\begin{ex}
\begin{equation*}
u(z)=-\frac{\left|\real z'\right|^2}{\log |z_1|}+h_n(|z_1|).
\end{equation*}
\end{ex}

This higher-dimensional solution has a similar singularity profile to the $2$-dimensional one. To verify that it is a pluripotential solution, we use approximation
$$u^{\varepsilon}(z)=-\frac{2|\real z'|^2}{\log \lp |z_1|^2+\varepsilon\rp }+h_n\big(\sqrt{|z_1|^2+\varepsilon}\big).$$ Since $h_n$ monotonically increases in $[0,1)$, $u^{\varepsilon}$ decreases to $u$ as $\varepsilon\rightarrow 0$. With similar calculations as in the $2$-dimensional case, we obtain 
\begin{equation*}
\begin{split}
\det\lp\pbp u^{\varepsilon}\rp=&1-\frac{\varepsilon}{|z_1|^2+\varepsilon}+\frac{2\varepsilon|\real z'|^2}{\lp |z_1|^2+\varepsilon\rp^2\lf-\log\lp |z_1|^2+\varepsilon\rp\rf^{n+1}}\\
 &+\frac{\varepsilon h'_n}{2(|z_1|^2+\varepsilon)^\frac{3}{2}\lf-\log\lp |z_1|^2+\varepsilon\rp\rf^{n-1}}.
\end{split}
\end{equation*}
As $h'_n(r)$ has the order $r\log^{n-1} r$ as $r\rightarrow 0$, $\det\lp\pbp u^{\varepsilon}\rp$ converges to $1$ in the locally $L^{1}(\log L^1)^p$ sense for any $p<n+1$. 

The proof that $u$ is a viscosity solution is almost the same as in the $2$-dimensional case. The key point is that $u$ attains its global infimum at the singular set $\{0\}\times\mathbb C^{n-1}$ and grows faster than any $C^2$ function away from the singular set.  

\section{Two families of solutions with continuous H\"older exponents}\label{sec4}

In this section, we prove Theorem \ref{main2} through the following higher-dimensional generalizations of Examples \ref{exa2} and \ref{exa3}. 
\begin{ex}\label{exa-2}
\begin{equation*}
v(z)=\frac{c_n}{(1-\beta)^{\frac{2}{n}}}|z'|^{2-\frac{2}{n}}\left(|z_1|^{\beta}+|z_1|^{2-\beta}\right)^\frac{2}{n},\quad  c_n=\frac{2^{-\frac{1}{n}}n^\frac{n+1}{n}}{n-1},\ \beta\in [0,1).
\end{equation*}
\end{ex}
\begin{ex}\label{exa-3}
\begin{equation*}
w(z)=|z_1|^{\frac{\gamma}{n-1}}|z'|^2+\frac{4}{(2-\gamma)^2}|z_1|^{2-\gamma}, \quad\gamma\in [0,2).
\end{equation*}	
\end{ex}

\begin{proof}[Proof of Theorem \ref{main2} by Example \ref{exa-2}.] The regularity of $v$ is straightforward to verify, so we will only confirm that $v$ is both a pluripotential and viscosity solution. 

{\it 1.\,Verification of pluripotential solution.} 

We prove this by approximation. For $\eps>0$, consider 
\begin{equation*}
v^{\eps}(z)=KG(|z'|^2+\varepsilon)H(|z_1|^2+\varepsilon),
\end{equation*}
where $K$ denotes $c_n(1-\beta)^{-\frac{2}{n}}$, 
$G(r)=r^\frac{n-1}{n}$, and $H(r)=(r^{\frac{\beta}{2}}+r^{1-\frac{\beta}{2}})^\frac{2}{n}.$ In the following, let subscripts $i$ and $j$ run from $2$ to $n$. Through direct computations, we have 
\begin{equation*}
v^{\eps}_1=KGH'\bar z_1,  \quad
v^{\eps}_i=KG'H\bar z_i;
\end{equation*}
\begin{equation*}
v^{\eps}_{1\bar 1}=KG\left(H''|z_1|^2+H'\right),\quad\  v^{\eps}_{i\bar 1}=KG'H'z_1\bar z_i ,\quad\  v^{\eps}_{i\bar j}=KH\left(G'\delta_{ij}+G''\bar z_i z_j\right).
\end{equation*}
It follows that
\begin{equation*}
\begin{split}
\det\lp \pbp v^{\eps}\rp=&K^n(HG')^{n-2}\big\{GG'H\left(H''|z_1|^2+H'\right)\\
&+\left[GG''H(H''|z_1|^2+H')-G'^2H'^2|z_1|^2\right]|z'|^2\big\}.
\end{split}
\end{equation*}
Substituting the expressions of $K$, $G$, and $H$, we finally obtain
\begin{equation*}
\begin{split}
\det\lp \pbp v^{\eps}\rp=&1+\frac{\varepsilon}{2(1-\beta)^2}\left[\beta Q^{\beta-2}+(2-\beta)Q^{-\beta}+2\beta(2-\beta)Q^{-1}\right]\\
&+\frac{\varepsilon}{2(n-1)(1-\beta)^2}\left[2+\beta^2 Q^{\beta-1}+(2-\beta)^2Q^{1-\beta}\right](|z'|^2+\eps)^{-1}\\
&+\frac{\varepsilon^2}{2(n-1)(1-\beta)}\left[\beta Q^{\beta-2}-(2-\beta)Q^{-\beta}\right](|z'|^2+\eps)^{-1},
\end{split}
\end{equation*}
where $Q=|z_1|^2+\varepsilon$. It is not hard to see that $\det\lp \pbp v^{\eps}\rp-1$ is dominated by  
$$\frac{\varepsilon}{(|z_1|^2+\varepsilon)^{2-\beta}}+\frac{\varepsilon}{(|z_1|^2+\varepsilon)^{1-\beta}(|z'|^2+\eps)}$$
when $|z_1|$ or $|z'|$ is small. Hence $\det(\partial\bar\partial v^\eps)\to 1$ in $L^{1+\delta}_{loc}$ sense for any $\delta<\frac{\beta}{1-\beta}$. 
Since $v^{\eps}$ converges decreasingly to $v$ in $L^{\infty}_{loc}(\bC^n)$, by Lemma \ref{convergence theorem}, we conclude that $(dd^c v)^n=\lim\limits_{\varepsilon\to 0}(dd^c v^{\eps})^n=4^nn!dV_{\bC^n}$. 

{\it 2.\,Verification of viscosity solution.} 

Based on the calculations for $v^\varepsilon$, we see that $v$ classically solves $\det(\partial\bar\partial v)=1$ on on $\bC^n\setminus \ls (z_1,z')\in\mathbb C^n\,\big|\,z_1=0\text{\ or\ }z'=0\rs$. Therefore, we only need to check points in $\{0\} \times\mathbb C^{n-1}$ and $\mathbb C\times \{0\}^{n-1}$, where $v$ takes value $0$. These two cases are very similar, so we only check the points in $\{0\} \times\mathbb C^{n-1}$.  

Suppose $p_0=(0,z'^0)$ and $\varphi$ is a $C^2$ function defined in a neighborhood of $p_0$ such that $u-\varphi$ attains its local maximum. Then, for any $p=(z_1,z')$ in this neighborhood of $p_0$, we have
$$\varphi(z_1,z')\ge \varphi(0,z'^0)+ \frac{c_n}{(1-\beta)^{\frac{2}{n}}}|z'|^{2-\frac{2}{n}}\left(|z_1|^{\beta}+|z_1|^{2-\beta}\right)^\frac{2}{n}.$$
If $z'^0\neq 0$, taking $z'=z'^0$, we get 
$$\varphi(z_1,z'^0)\ge \varphi(0,z'^0)+ \frac{c_n}{(1-\beta)^{\frac{2}{n}}}|z'^0|^{2-\frac{2}{n}}\left(|z_1|^{\beta}+|z_1|^{2-\beta}\right)^\frac{2}{n}.$$
If $z'^0=0$, we have 
$$\varphi(z_1,z')\ge \varphi(0,0)+ \frac{c_n}{(1-\beta)^{\frac{2}{n}}}|z'|^{2-\frac{2}{n}}\left(|z_1|^{\beta}+|z_1|^{2-\beta}\right)^\frac{2}{n}.$$
Either case contradicts the $C^2$ regularity of $\varphi$. Therefore, there is no upper test function, which proves that $u$ is a viscosity subsolution. 

Now, suppose $p_0=(0,z^0_2)$ and $\varphi$ is a $C^2$ function defined in a neighborhood of $p_0$ such that $u-\varphi$ attains its local minimum. Then, for any $p=(z_1,z')$ in this neighborhood of $p_0$, we have
$$\varphi(z_1,z')-\varphi(0,z'^0)\le \frac{c_n}{(1-\beta)^{\frac{2}{n}}}|z'|^{2-\frac{2}{n}}\left(|z_1|^{\beta}+|z_1|^{2-\beta}\right)^\frac{2}{n}.$$
Taking $z_1=0$, we get
\begin{equation*}
\varphi(0,z')\leq \varphi(0,z'^0).
\end{equation*}
This implies that the $(n-1)$-order submatrix $\p_{z'}\p_{\bar z'}\varphi(p_0)$ is semi-negative. If $\pbp\varphi(p_0)\geq 0$, then $\pbp\varphi(p_0)$ has zero eigenvalues. Thus $\det (\pbp\varphi(p_0))=0$ and $F(\pbp\varphi(p_0))=1$, where $F$ is defined in \eqref{F}. Otherwise, $F(\pbp\varphi(p_0))=+\infty$. In either case, we have $F(\pbp\varphi(p_0))\geq 0$. This proves that $u$ is a viscosity supersolution.  
\end{proof}

\begin{proof}[Verification of Example \ref{exa-3}.] The verification that $w$ is a viscosity solution is very similar to Examples \ref{exa1} and \ref{exa-3}, so we only check that $w$ is a pluripotential solution. 

As in the proofs for the previous two examples, we use approximation to establish this. For $\eps>0$, let
$$w^{\eps}=\lp |z_1|^2+\eps\rp^{\frac{\gamma}{2(n-1)}}|z'|^2+\frac{4}{(2-\gamma)^2}\lp |z_1|^2+\eps\rp^{\frac{2-\gamma}{2}}.$$
Clearly, $w^{\varepsilon}$ converges decreasingly and locally uniformly to $w$ as $\varepsilon\to 0$. Let subscripts $i$ and $j$ run from $2$ to $n$. By direct computations, we have
\begin{equation*}
\begin{split}
w^{\eps}_{1\bar 1}=&\frac{\gamma^2}{4(n-1)^2}|z'|^2|z_1|^2\lp |z_1|^2+\eps\rp^{\frac{\gamma}{2(n-1)}-2}+\frac{\gamma\varepsilon}{2(n-1)}|z'|^2\lp |z_1|^2+\eps\rp^{\frac{\gamma}{2(n-1)}-2}  \\
&+\lp |z_1|^2+\eps\rp^{-\frac{\gamma}{2}} +\frac{\gamma\varepsilon}{2-\gamma}|z_1|^2\lp |z_1|^2+\eps\rp^{-\frac{\gamma}{2}-1},
\end{split}
\end{equation*}
\begin{equation*}
 w^{\eps}_{1\bar i}=\frac{\gamma}{2(n-1)}\lp |z_1|^2+\eps\rp^{\frac{\gamma}{2(n-1)}-1}\bar z_1 z_i,\quad\ w^{\eps}_{i\bar j}= \lp |z_1|^2+\eps\rp^{\frac{\gamma}{2(n-1)}}\delta_{ij}.
\end{equation*}
It follows that
\begin{equation*}
\begin{split}
\det\left(\partial\bar\partial w^\varepsilon\right)=&1+\eps\left[ \frac{\gamma |z'|^2}{2(n-1)}\lp |z_1|^2+\eps\rp^{\frac{n\gamma}{2(n-1)}-2}+\frac{\gamma |z_1|^2}{2-\gamma}\lp |z_1|^2+\eps\rp^{-1}\right].
\end{split}
\end{equation*}
Hence, $\det\left(\partial\bar\partial w^\varepsilon\right)\rightarrow 1$ in $L^{1+\delta}_{loc}(\mathbb C^n)$ for any $0<\delta< \frac{n\gamma}{ 2n-2-\gamma n}$ when $\gamma<\frac{2(n-1)}{n}$ and in $L^{p}_{loc}(\bC^n)$ for any $p>0$ when $\gamma\geq \frac{2(n-1)}{n}$. 
Since $w^{\eps}$ converges decreasingly to $w$ in $L^{\infty}_{loc}(\bC^n)$, we conclude that $(dd^c w)^n=\lim\limits_{\varepsilon\to 0}(dd^c w^{\eps})^n=4^nn!dV_{\bC^n}$. 

\end{proof}

\vskip 5pt

{\it Acknowledgement.} The authors would like to thank Weiming Shen, Yuguang Shi, Xujia Wang, Yu Yuan, Bin Zhou, and Xiaohua Zhu for helpful discussions.

\end{document}